\documentclass[11pt,reqno]{amsart}
\usepackage{amsmath,amssymb}

 \makeatletter
 \evensidemargin  \oddsidemargin
 \makeatother

\newcommand{\N}{\mbox{$\mathbb{N}$}}

\begin{document}
\thispagestyle{empty}
 \setcounter{page}{1}

\begin{center}
{\large\bf On the Mazur--Ulam theorem  in fuzzy normed spaces
\vskip.25in

{\bf M. Eshaghi Gordji} \\[2mm]

{\footnotesize Department of Mathematics,
Semnan University,\\ P. O. Box 35195-363, Semnan, Iran\\
[-1mm] Tel:{\tt 0098-231-4459905} \\
[-1mm] Fax:{\tt 0098-231-3354082} \\
[-1mm] e-mail: {\tt madjid.eshaghi@gmail.com}}

{\bf N. Ghobadipour} \\[2mm]

{\footnotesize Department of Mathematics,
Semnan University,\\ P. O. Box 35195-363, Semnan, Iran\\
[-1mm] e-mail: {\tt ghobadipour.n@gmail.com }} }
\end{center}
\vskip 5mm

\noindent{\footnotesize{\bf Abstract.} In this paper, we  establish
the Mazur--Ulam theorem in the fuzzy strictly convex real normed spaces.\\



  \newtheorem{df}{Definition}[section]
  \newtheorem{rk}[df]{Remark}
   \newtheorem{lem}[df]{Lemma}
   \newtheorem{thm}[df]{Theorem}
   \newtheorem{pro}[df]{Proposition}
   \newtheorem{cor}[df]{Corollary}
   \newtheorem{ex}[df]{Example}

 \setcounter{section}{0}
 \numberwithin{equation}{section}

\vskip .2in

\begin{center}
\section{Introduction}
\end{center}

In 1984, Katsaras \cite{Kat} defined a fuzzy norm on a linear space
and at the same year Wu and Fang \cite{Co} also introduced a notion
of fuzzy normed space and gave the generalization of the Kolmogoroff
normalized theorem for  fuzzy topological linear space. In
\cite{Bi}, Biswas defined and studied fuzzy inner product spaces in
linear space. Since then some mathematicians have defined fuzzy
metrics and norms on a linear space from various points of view
\cite{Ba,Fe,Kri,Shi,Xi}. In 1994, Cheng and Mordeson introduced a
definition of fuzzy norm on a linear space in such a manner that the
corresponding induced fuzzy metric is of Kramosil and Michalek type
\cite{Kra}. In 2003, Bag and Samanta \cite{Ba} modified the
definition of Cheng and Mordeson \cite{Che} by removing a regular
condition. They also established a decomposition theorem of a fuzzy
norm into a family of crisp norms and investigated some properties
of fuzzy norms (see \cite{Bag2}). Following \cite{Bag1}, we give the
employing notion of a fuzzy
norm.\\
Let X be a real linear space. A function $N : X \times \Bbb R
\longrightarrow [0,1]$ (the so--called fuzzy subset) is said to be a
fuzzy norm on X if for all $x, y \in X$ and all $a,b \in \Bbb
R$:\\
$(N_1)~~N(x,a)=0$ for $a\leq 0;$\\
$(N_2)~~x=0$ if and only if $N(x,a)=1$ for all $a>0;$\\
$(N_3)~~N(ax,b)=N(x,\frac{b}{|a|})$ if $a\neq0$;\\
$(N_4)~~N(x+y,a+b)\geq min \{N(x,a), N(y,b)\};$\\
$(N_5)~~N(x,.)$ is non--decreasing function on $\Bbb R$ and $\lim
_{a \to \infty} N(x,a)=1;$\\
$(N_6)~~$ For $x\neq 0,$ $N(x,.)$ is (upper semi) continuous on
$\Bbb R.$\\
The pair $(X,N)$ is called a fuzzy normed linear space. One may
regard $N(x,a)$ as the truth value of the statement "the norm of $x$
is less than or equal to the real number $a$".\\
\begin{ex}
Let $(X,\|.\|)$ be a normed linear space. Then $$N(x,a)=\left\{%
\begin{array}{ll}
   \frac{a}{a+\|x\|}, & a>0~~~, x \in X, \\
    0,~~~ & a\leq0, x \in X \\
\end{array}%
\right.    $$ is a fuzzy norm on X.
\end{ex}
Let $(X,N)$ be a fuzzy normed linear space. Let $\{x_n\}$ be a
sequence in X. Then $\{x_n\}$ is said to be convergent if there
exists $x \in X$ such that $\lim_{n \to \infty}N(x_n-x,a)=1$ for
all $a>0.$ In that case, $x$ is called the limit of the sequence
$\{x_n\}$ and we denote it by $N - lim_{n \to \infty} ~ x_n = x.$
A sequence $\{x_n\}$ in X is called Cauchy if for each $\epsilon >
0$ and each $a_0$ there exists $n_0$ such that for all $n\geq n_0$
and all $p > 0,$ we have $N(x_{n+p} - x_n, a) > 1$ - $\epsilon$.
It is known that every convergent sequence in  fuzzy normed space
is Cauchy. If each Cauchy sequence is convergent, then the fuzzy
norm is said to be complete and the fuzzy normed space is called a
fuzzy Banach space.\\
\begin{df}
A fuzzy normed space is called {\it strictly convex} if
$N(x+y,a+b)=min\{N(x,a),\\N(y,b)\}$ and $N(x,a)=N(y,b)$ imply
$x=y$ and $a=b.$
\end{df}
\begin{df}
Let $(X,N)$ and $(Y,N)$ be two fuzzy normed spaces. We call
$f:(X,N) \to (Y,N)$ a {\it fuzzy isometry} if
$N(x-y,a)=N(f(x)-f(y),a)$ for all $x,y \in X$ and all $a>0.$
\end{df}
\begin{df}
Let $X$ be a real linear space and $x,y,z$ mutually disjoint
elements of $X.$ Then $x,y$ and $z$ are said to be {\it collinear}
if $y-z=t(x-z)$ for some real number $t.$
\end{df}
The theory of isometry had its beginning in the important paper by
Mazur and Ulam \cite{Maz} in 1932.\\ In \cite{Maz}, Mazur and Ulam
proved the following theorem.
\begin{thm} (Mazur--Ulam)
Every isometry of a real normed linear space onto a real normed
linear space is a linear mapping up to translation.
\end{thm}
Also, Baker \cite{Bak} investigated the  Mazur--Ulam problem and
obtained the following result.
\begin{thm}\cite{Bak}
Let $X$ and $Y$ be real normed linear spaces and suppose that $Y$
is strictly convex. If $f:X \to Y$ is an isometry, then $f$ is
affine.
\end{thm}
Recently, V$\ddot{a}$is$\ddot{a}$l$\ddot{a}$ \cite{Vai} proved
that every bijective isometry
$f: X \to Y$ between normed linear spaces is affine.\\
In this paper, by using ideas of \cite{Bak, MS}, we investigate a
Mazur--Ulam type problem in the framework of strictly convex fuzzy
normed spaces.
\section{ Mail result }
 In this section we  establish
the Mazur--Ulam theorem in the fuzzy strictly convex real normed
spaces. We need the following lemma to prove the main theorem of our
paper.
\begin{lem} Let $X$ be a fuzzy normed space which is strictly
convex and let $a,b \in X$ and $s>0.$ Then $x=\frac{a+b}{2}$ is
the unique member of $X$ satisfying $$N(a-x,s)=N(a-b,2s)$$ and
$$N(b-x,s)=N(a-b,2s).$$
\end{lem}
\begin{proof}
There is nothing to prove if $a=b.$ Let $a\neq b.$ Then by
$(N_3),$ we have $$N(a-x,s)=N(a-\frac{a+b}{2},s)=N(a-b,2s)$$ and
$$N(b-x,s)=N(b-\frac{a+b}{2},s)=N(a-b,2s).$$ To show the
uniqueness, assume that $u$ and $v$ are two elements of $X$
satisfying
$$N(a-u,s)=N(a-v,s)=N(b-u,s)=N(b-v,s)=N(a-b,2s).\hspace{2cm}$$
Then
$$N(a-\frac{u+v}{2},s)\geq min \{N(a-u,s),N(a-v,s)\}=N(a-b,2s) \hspace{2cm}(2.1)$$
and $$N(b-\frac{u+v}{2},s)\geq min
\{N(b-u,s),N(b-v,s)\}=N(a-b,2s).\hspace{2cm}(2.2)$$ If both of
inequalities $(2.1)$ and $(2.2)$ were strict we would have
\begin{align*}
N(a-b,2s)&=N(a-\frac{u+v}{2}+\frac{u+v}{2}-b,2s)\\
&\geq min
\{N(a-\frac{u+v}{2},s),N(b-\frac{u+v}{2},s)\}\\
&>N(a-b,2s),
\end{align*}
a contradiction. So at least one of the equalities holds in $(2.1)$
and $(2.2).$ Without loss of generality assume that equality holds
in $(2.1).$ Then $$N(a-\frac{u+v}{2},s)= min
\{N(a-u,s),N(a-v,s)\}.$$ By the strict convexity, we obtain
$N(a-u,s)=N(a-v,s)$ that is $u=v.$
\end{proof}
\begin{thm}
 Let $X$ and $Y$ be  fuzzy  real normed
spaces and let $Y$ be  strictly convex. Suppose  $f: X \to Y$ is a
fuzzy isometry satisfies  $f(a),f(b)$ and $f(c)$ are collinear when
$a,b$ and $c$ are collinear. Then $f$ is affine.
\end{thm}
\begin{proof}
We prove that $f-f(0)$ is linear. Let $h(a)=f(a)-f(0).$ Then $h$ is
a fuzzy isometry and $h(0)=0.$ Also we have
$$N(h(\frac{a+b}{2})-h(a),s)=N(\frac{a+b}{2}-a,s)=N(a-b,2s)$$  and similarly
$$N(h(\frac{a+b}{2})-h(b),s)=N(\frac{a+b}{2}-b,s)=N(a-b,2s)$$
for all $a,b \in X$ and all $s>0.$ It follows from Lemma $2.1$ that
$$h(\frac{a+b}{2})=\frac{1}{2}h(a)+\frac{1}{2}h(b).$$ Since
$h(0)=0,$ we can easily show that $h$ is additive. It follows that
$h$ is $\Bbb Q-$ linear. Let $r \in \Bbb R$ with $r\neq 1$ and $a
\in X.$ Since $h(0)=0$ and $h(0)$ is collinear, then  there exists a
real number $r'$ such that $h(ra)=r'h(a).$ Now, we prove that
$r=r'.$ Since $a$ and $b$ are collinear, then $a\neq b.$ Hence,
\begin{align*}
N(r(a-b),s)=N(h(ra)-h(rb),s)=N(r'(h(a)-h(b)),s)=N(r'(a-b),s).
\end{align*}
By the strict convexity we obtain $r(a-b)=r'(a-b).$ Thus
$h(ra)=rh(a)$ for all $a \in X$ and all $r \in \Bbb R.$ Hence, $f$
is affine.
\end{proof}

\section{Conclusion} We  study the
notion of fuzzy strictly convex normed spaces, and then we establish
the Mazur--Ulam theorem in the fuzzy strictly convex real normed
spaces.

{\small


}
\end{document}